\documentclass[12pt,twoside,a4paper]{amsart}
\usepackage{amssymb}
\usepackage{graphicx}
\usepackage{amsmath}
\usepackage{amsmath}
\usepackage{amsthm,amssymb,latexsym}
\usepackage{t1enc}
\usepackage{mathrsfs}

\usepackage{graphicx}
\usepackage{xypic}

\date{\today}
\thanks{The author was partially supported by the Swedish Research Council} 

\usepackage[latin1]{inputenc}
\usepackage[T1]{fontenc}

\def\deg{\text{deg}\,}

\def\dbar{\bar\partial}

\def\R{{\mathbb R}}
\def\C{{\mathbb C}}
\def\P{{\mathbb P}}

\def\O{{\mathcal O}}

\def\Re{{\rm Re\,  }}

\def\U{{\mathcal U}}

\def\codim{\text{codim}\,}

\def\supp{\text{supp}\,}

\def\1{\mathbf 1}

\def\Z{{\mathbb Z}}
\def\N{{\mathbb N}}

\def\J{{\mathcal J}}

\def\vol{\text{Vol}}
\def\np{\mathcal{NP}}

\def\be{\begin{equation}}
\def\ee{\end{equation}}

\def\J{{\mathcal J}}

\def\a{{\mathfrak a}}

\def\poly{{\mathcal P}}
\def\Q{{\mathcal P}}
\def\Qq{{\mathcal Q}}
\def\pdeg{\text{deg}}
\def\set{{\mathcal S}}

\newtheorem{thm}{Theorem}[section]
\newtheorem{lma}[thm]{Lemma}

\theoremstyle{definition}

\theoremstyle{remark}

\newtheorem{preremark}[thm]{Remark}
\newtheorem{preex}[thm]{Example}

\newenvironment{remark}{\begin{preremark}}{\qed\end{preremark}}
\newenvironment{ex}{\begin{preex}}{\qed\end{preex}}

\numberwithin{equation}{section}

\begin{document}

\title[Sparse effective membership problems via residue currents ]{Sparse effective membership problems \\ via residue currents} 

\date{\today}

\author{Elizabeth Wulcan}

\address{Dept of Mathematics, University of Michigan, Ann Arbor \\ MI 48109-1043\\ USA}

\email{wulcan@umich.edu}




\maketitle

\begin{abstract}
We use residue currents on toric varieties to obtain bounds on the degrees of solutions to polynomial ideal membership problems. Our bounds depend on (the volume of) the Newton polytope of the polynomial system and are therefore well adjusted to sparse polynomial systems. We present sparse versions of Max N\"other's $AF+BG$ Theorem, Macaulay's Theorem, and Koll{\'a}r's Effective Nullstellensatz, as well as recent results by Hickel and Andersson-G\"otmark. 
\end{abstract}

\section{Introduction}\label{intro}

Residue currents are generalizations of one complex variable residues and can be thought of as currents representing ideals of holomorphic functions or polynomials. 
The purpose of this paper is to investigate how residue currents on toric varieties can be used to obtain effective solutions to polynomial ideal membership problems. 

Let $F_1, \ldots, F_m$, and $\Phi$ be polynomials in $\C [z_1,\ldots, z_n]$. Assume that $\Phi$ vanishes on the zero set $V_F$ of the $F_j$. Then \emph{Hilbert's Nullstellensatz} asserts that there are polynomials $G_1,\ldots, G_m$ such that 
\begin{equation}\label{hilbert}
\sum_{j=1}^m F_j G_j = \Phi ^\nu
\end{equation}
for some integer $\nu$ large enough. Much attention has recently been paid to the problem of bounding the complexity of the solutions to \eqref{hilbert}, starting with the breakthrough work of Brownawell ~\cite{Br}. For example, one can ask for bounds of $\nu$ and the degrees of the $G_j$ in terms of the degrees of the $F_j$. The optimal result in this direction was obtain by Koll{\'a}r ~\cite{Kollar}: 

\noindent
\emph{Assume that $\deg F_j \leq d\neq 2$. Then one can find $G_j$ so that \eqref{hilbert} holds for some $\nu \leq d^{\min(m,n)}$ and
\begin{equation}\label{kollupp}
\deg (F_j G_j) \leq (1 + deg \Phi )d^{\min (m,n)}.
\end{equation}
} 
The restriction $d\neq 2$ was removed by Jelonek, ~\cite{Jelonek}, for $m\leq n$. For $m\geq n+1$ Sombra, \cite{Sombra}, proved that one can find $G_j$ that satisfy $\deg (F_j G_j) \leq (1 + deg \Phi )2^{n+1}$. 

Koll{\'a}r's and Jelonek's result are sharp; the original statements also take into account different degrees of the $F_j$. In many cases, however, one can do much better.  Classical results due to Max N\"other ~\cite{Not} and Macaulay ~\cite{Mac} show that the bounds can be substantially improved if (the homogenizations of) the $F_j$ have no zeros at infinity. 

Another situation in which one can improve Koll{\'a}r's result is when the system of polynomials is \emph{sparse}, meaning that its Newton polytope has small volume. Recall that the \emph{support} $\supp F$ of a Laurent polynomial $F=\sum_{\alpha\in\Z^n} c_\alpha z^\alpha =\sum_{\alpha\in\Z^n} c_\alpha z_1^{\alpha_1}\cdots z_n^{\alpha_n}$ in $\C[z_1^{\pm 1},\ldots, z_n^{\pm 1}]$ is defined as $\supp F=\{\alpha\in\Z^n \text{ such that } c_\alpha \neq 0\}$ and that the \emph{Newton polytope} $\np(F_1,\ldots, F_m)$ of the system of polynomials $F_1,\ldots, F_m$ is the convex hull of $\bigcup_j \supp F_j$ in $\R^n$. In particular, a polynomial of degree $d$ has support in $d\Sigma^n$, where $\Sigma^n$ is the $n$-dimensional simplex in $\R^n$ with the origin and the unit lattice points $e_1=(1,0,\ldots,0), e_2=(0,1,0,\ldots,0), \ldots, e_n=(0,\ldots, 0, 1)$ as vertices. The \emph{normalized volume} $\vol(\set)$ of a convex set $\set$ in $\R^n$ is $k!$ times the Euclidean volume of $\set$, where $k$ is the dimension of $\set$, so that $\vol(\Sigma^n)=1$. A \emph{lattice polytope} is a polytope in $\R^n$ with vertices in $\Z^n$. Sombra ~\cite{Sombra} proved the following using techniques from toric geometry:

\noindent
\emph{
Let $\Q$ be a lattice polytope that contains $\np=\np(F_1, \ldots, F_m, 1, z_1, \ldots, z_n)$. Then there are polynomials $G_j$ that satisfy \eqref{hilbert} for $\nu\leq n^{n+2}\vol(\Q)$ and 
\begin{equation}\label{sombrastod}
\supp (F_j G_j)\subseteq (1+\deg\Phi)n^{n+3} \vol (\Q) \Q.
\end{equation}
In particular, if $\deg F_j \leq d$, then
\begin{equation}\label{sombradeg}
\deg (F_j G_j) \leq (1+\deg\Phi)n^{n+3} \vol (\np) d.
\end{equation} 
}
In general the bound \eqref{sombradeg} is less sharp than Koll{\'a}r's bound, but if $d$ is large compared to $n$ and $\vol(\Q)$ is small compared to $\vol (d\Sigma^n)=d^n$, then \eqref{sombradeg} is sharper than \eqref{kollupp}. 

The main ingredient in Sombra'a proof is an effective Nullstellensatz for arithmetically  Cohen-Macaulay varieties, \cite[Lemma~1.1]{Sombra}. This result was later extended to  general varieties by Koll{\'a}r, ~\cite{K2}, Ein-Lazarsfeld, ~\cite{EL}, and Jelonek, ~\cite{Jelonek}. Combining their results and Sombra's techniques, \eqref{sombrastod} can be substantially improved; in many cases one can get rid of the factor $n^{n+3}$, \cite{Sombra2}. 
For example, if $F_1,\ldots, F_n$ lack common zeros then one can solve \eqref{hilbert} with $\Phi=1$ and 
\begin{equation*}
\supp (F_jG_j)\subseteq \vol (\Q) \Q,
\end{equation*}
as follows using Jelonek's Nullstellensatz, \cite{Jelonek}. 
In \cite[Example~2]{EL}, due to Rojas, the special case when $\Q$ is a product of simplices is considered.

\smallskip

Residue currents have been used as a tool to solve polynomial membership problems by several authors, see, for example, \cite{BGVY}. In this paper we extend the ideas developed by Andersson ~\cite{A3} and Andersson-G\"otmark ~\cite{AG}, who used residue currents on complex projective space $\P^n$ to obtain effective solutions. We consider residue currents on general toric compactifications of $\C^n$ in order to obtain sparse effective results. Given a lattice polytope $\poly$ one can construct a toric variety $X_\poly$ and a line bundle $\O(D_\poly)$ on $X_\poly$ whose global sections correspond precisely to polynomials with support in $\poly$, see Section ~\ref{toric}. The toric variety $X_\poly$ is smooth if for each vertex $v$ of $\poly$ the smallest integer normal vectors of the facets of $\poly$ containing $v$ form a base for the $\Z^n$, see ~\cite[p.~29]{F}. We then say that the lattice polytope $\poly$ is \emph{smooth} with respect to the lattice $\Z^n$, see ~\cite{CS}. 

The following sparse version of Macaulay's Theorem \cite{Mac} is due to Castryck-Denef-Vercauteren \cite{CDV}.
\begin{thm}\label{mac}
Let $F_1, \ldots, F_m$ be polynomials in $\C [z_1,\ldots, z_n]$ and let $\Q$ be a lattice polytope that contains the Newton polytope of $F_1,\ldots F_m$. Assume that the $F_j$ have no common zeros neither in $\C^n$ nor at infinity. 
Then there are polynomials $G_j$ that satisfy 
\begin{equation}\label{ett}
\sum_{j=1}^m F_j G_j = 1
\end{equation} 
and
\begin{equation}\label{lasse}
\supp (F_j G_j)\subseteq (n+1) \Q.
\end{equation}
\end{thm}
We will specify in Section ~\ref{macresults} how no common zeros at infinity should be interpreted.

Macaulay's Theorem, ~\cite{Mac}, corresponds to the case when $\Q=d\Sigma^n$, that is, $\deg F_j\leq d$. Then \eqref{lasse} reads $\deg (F_j G_j) \leq (n+1) d$. Macaulay's original result is in fact slightly stronger; we refer to Section ~\ref{macresults} for an exact statement. In the special case when $\Q$ is of the form $\Q=d\Sigma^n$ or more generally of the form
\begin{equation}\label{product}
\Q=d_1\Sigma^{n_1}\times\cdots\times d_r \Sigma^{n_r}
\end{equation}
we get a slightly sharper bound than \eqref{lasse}, see Theorem ~\ref{macproduct}; in particular, we get back Macaulay's result. Observe that $\supp F\subseteq \Q$, where $\Q$ is given by \eqref{product}, means that the degree in the first $n_1$ variables are bounded by $d_1$, the degree in the next $n_2$ variables are bounded by $d_2$, etc. 

\smallskip

Our next result is a sparse version of Max N\"other's $AF+BG$ Theorem, ~\cite{Not}. 
\begin{thm}\label{max}
Let $F_1,\ldots, F_m$, and $\Phi$ be polynomials in $\C [z_1,\ldots, z_n]$ and let $\Q$ be a smooth and ``large'' polytope that contains the origin and the support of $\Phi$ and the coordinate functions $z_1, \ldots, z_n$. Assume that $\Phi\in (F_1, \ldots, F_m)$ and moreover that  the codimension of the zero set of the $F_j$ is $m$ and that it has no component contained in the variety at infinity. Then there are polynomials $G_j$ such that
\begin{equation}\label{ideal}
\sum F_j G_j = \Phi
\end{equation}
and 
\begin{equation*}
\supp (F_j G_j)\subseteq \Q.
\end{equation*}
\end{thm}
It will be specified in Section ~\ref{maxresults} what we mean by that the zero set of the $F_j$ has no component contained in the variety at infinity and by that the polytope is ``large''; in particular, $\poly$ is large if it is of the form $(n+1)$ times a lattice polytope. Theorem ~\ref{max} also holds if $\Q$ is of the form \eqref{product}. In particular, if $m=n$ and $\Q=(\deg \Phi)\Sigma^n$ we get back N\"other's original result ~\cite{Not}: 

\noindent
\emph{
Assume that the zero-set of $F_1,\ldots, F_n$ is discrete and contained in $\C^n$ and that $\Phi\in(F_1,\ldots, F_n)$. Then, there are  $G_j$ that satisfy \eqref{ideal} and $\deg(F_j G_j)\leq \deg\Phi$.
}

If $\Q=(\deg \Phi) \Sigma^n$ but we drop the condition $m=n$, then the corresponding result appeared as Theorem 1.2 in ~\cite{A3}. 

\smallskip

In general, the $F_j$ have common zeros at infinity. The following is a sparse version of a result by Andersson-G\"otmark, \cite[Theorem~1.3]{AG}, which generalizes N\"other's Theorem to the situation when there are no restriction on the zeros of the $F_j$ at infinity. 

\begin{thm}\label{agsats}
Let $F_1,\ldots, F_m$, and $\Phi$ be polynomials in $\C [z_1,\ldots, z_n]$, let $\Q$ be a smooth polytope that contains the origin and the Newton polytope of $F_1, \ldots, F_m, z_1, \ldots, z_n$, and let $a$ denote the minimal side length of $\Q$. Assume that the codimension of the common zero set of $F_1, \ldots, F_m$ in $\C^n$ is $m$, that $\Phi\in(F_1, \ldots, F_m)$, and that $\supp \Phi \subseteq e \Q$, where $e\Q$ is a lattice polytope. Then there are polynomials $G_j$ that satisfy \eqref{ideal} and 
\begin{equation}\label{aguppsk}
\supp (F_j G_j)\subseteq \lceil e + m \vol(\Q)/a \rceil \Q.
\end{equation}
\end{thm} 
By the minimal side length of $\Q$ we mean the length of the shortest edge of $\Q$. For example, if $\Q=d\Sigma$, then $a=d$. Thus with $\Q=d\Sigma^n$
\eqref{aguppsk} reads
\[ 
\deg (F_j G_j)\leq (\deg \phi/d+m d^n/d)d=(\deg\phi +m d^{n}),
\]
which is Andersson-G\"otmark's result in the case when the degrees of the $F_j$ are bounded by $d$ and $m=n$. Their result is more precise; in particular, it allows for the $F_j$ to have different degrees and $d^n$ in the estimate should be replaced by $d^{\min(m,n)}$. 

\smallskip

Recall that $\Phi$ lies in the \emph{integral closure} $\overline{(F)}$ of $(F)=(F_1,\ldots, F_m)$ if $\Phi$ satisfies a monic equation $\Phi^r+H_1 \Phi^{r-1} + \cdots + H_r=0$, where $H_j\in(F)^j$ for $1\leq j\leq r$, or, equivalently, if $\Phi$ locally satisfies $|\Phi|\leq C|F|$, where $|F|^2=|F_1|^2+\cdots + |F_m|^2$. If $\Phi\in \overline{(F)}$, then the Brian{\c c}on-Skoda Theorem, ~\cite{BS}, asserts that one can solve \eqref{hilbert} with $\nu=\min(m,n)$. The following is a sparse versions of an effective Brian{\c c}on-Skoda Theorem due to Hickel ~\cite[Theorem~1.1]{H}, see also Ein-Lazarsfeld ~\cite[p.~430]{EL}. 

\begin{thm}\label{hickelsats}
Let $F_1,\ldots, F_m$, and $\Phi$ be polynomials in $\C [z_1,\ldots, z_n]$, let $\Q$ be a smooth polytope that contains the origin and the Newton polytope of $F_1, \ldots, F_m, z_1, \ldots, z_n$, and let $a$ denote the minimal side length of $\Q$. Assume that $\Phi$ is in the integral closure of $(F_1, \ldots, F_m)$ and that $\supp \Phi \subseteq e \Q$, where $e\Q$ is a lattice polytope. Then there are polynomials $G_j$ such that 
\begin{equation}\label{bsekv}
\sum_{j=1}^m F_j G_j = \Phi^{\min(m,n)}
\end{equation}
and 
\begin{equation}\label{hickeluppsk}
\supp (F_j G_j)\subseteq \max\left (\lceil \min(m,n) (e+\vol (\Q)/a)\rceil, \min(m,n+1)\right ) \Q.
\end{equation}
\end{thm}
In most cases $\lceil \min(m,n) (e+\vol (\Q)/a)\rceil$ is much larger than $\min(m,n+1)$. In fact, $\min(m,n+1)$ is the largest only when $\Q=\Sigma^n$ and $e=0$.

If $\Q=d\Sigma^n$, then \eqref{hickeluppsk} reads 
$\deg(F_j G_j) \leq \min(m,n) (\deg \Phi+d^n)$, which is precisely Hickel's result, provided $m\geq n$. Hickel's original formulation is more precise, taking into account different degrees of the $F_j$; also, $d^n$ in the estimate should be replaced by $d^{\min(m,n)}$.

\smallskip

Finally, we have the following sparse Nullstellensatz. 
\begin{thm}\label{kollarsats}
Let $F_1,\ldots, F_m$, and $\Phi$ be polynomials in $\C [z_1,\ldots, z_n]$, let $\Q$ be a smooth polytope that contains the origin and the Newton polytope of $F_1,\ldots, F_m, z_1, \ldots, z_n$, and let $a$ denote the minimal side length of $\Q$. Assume that $\Phi$ vanishes on the zero set of the $F_j$ and that $\supp\Phi\subseteq e \Q$, where $e\Q$ is a lattice polytope. Then there are polynomials $G_j$ such that 
\begin{equation}\label{kollarekv}
\sum F_j G_j = \Phi^{\min(m,n) \vol(\Q)}
\end{equation}
and 
\begin{equation}\label{hus}
\supp (F_j G_j) \subseteq  
\max (\lceil \min(m,n) (1/a+e) \vol (\Q)\rceil, \min (m,n+1)) \Q.
\end{equation}
\end{thm}
Note that in most cases $\lceil \min(m,n) (e+1/a) \vol (\Q) \rceil$ is much larger than $\min(m,n+1)$. As above, $\min(m,n+1)$ is the largest only if $\Q=\Sigma^n$ and $e=0$. 

If  $\Q=d\Sigma^n$, then \eqref{hus} reads 
$\deg (F_j G_j) \leq \min (m,n)(1+\deg \Phi) d^n$. 
Moreover the exponent in \eqref{kollarekv} is $\min(m,n)d^n$, so if $n\geq m$ we get back Koll{\'a}r's result modulo a factor $n$ in the exponent $\nu$ in \eqref{hilbert} and in the degree estimate \eqref{kollupp}. Because of the factor $1/a$, Theorem ~\ref{kollarsats} slightly improves Sombra's result when $\Q$ is smooth. Also from a modified version of Theorem ~\ref{kollarsats} we recover Rojas' example ~\cite[Example~2]{EL}, see Section ~\ref{inftyresults}.

We will provide a proof of Theorem ~\ref{kollarsats} using residue currents. However, this result should be possible to conclude from Ein-Lazarsfeld's Geometric Effective Nullstellensatz \cite{EL}, cf. ~\cite[Example~2]{EL}, although we get a slightly better coefficient: a factor $\min(m,n)$ instead of $\min(m,n+1)$.

\smallskip

Let us sketch the idea of the proofs of our results. A standard way of reformulating the kind of division problems we consider is the following. There are polynomials $G_j$ that satisfies \eqref{hilbert} and $\supp (F_j G_j)\subseteq c \poly$ if and only if there are sections $g_j$ of line bundles $\O(D_{(c-1)\poly})$ over $X_\poly $ such that 
\begin{equation}\label{mja}
\sum_{j=1}^m f_j g_j = \psi,
\end{equation}
where $f_j$ and $\psi$ are sections of line bundles $\O(D_\poly)$ and $\O(D_{c\poly})$ over $X_\poly$ corresponding to $F_j$ and $\Phi^\nu$, respectively. In ~\cite{A} it was shown that $\psi$ solves \eqref{mja} locally on $X_\poly$ if $\psi$ annihilates the so-called Bochner-Martinelli residue current $R^f$ of $f_1,\ldots, f_m$, see Section ~\ref{rescurr}. 
To obtain a global solution to \eqref{mja} the constant $c$ has to be large enough so that certain Dolbeault cohomology on $X_\poly$ vanishes. By analyzing when these conditions are satisfied we obtain our results. In general, $\psi$ annihilates $R^f$ if it vanishes to high enough order along the zero set $V_f$ of $f_1,\ldots, f_m$; this is used to prove Theorems ~\ref{kollarsats} and ~\ref{hickelsats}. Ein-Lazarsfeld ~\cite{EL} (as well as Brownawell ~\cite{Br}) used Skoda's Theorem \cite{Skoda} to obtain analogous results. If the codimension of $V_f$ is $m$, we have a more refined estimate of when $R^f$ is annihilated, which makes it possible to get results such as Theorems \ref{max} and \ref{agsats}.

The somewhat unsatisfactory assumption in most of our results that the polytope $\Q$ is smooth is explained by the fact that the use of residue current techniques limits us to work on smooth toric varieties, cf. Remark \ref{dagas}. The Bochner-Martinelli residue current can actually be defined also on singular varieties; it will however not have as nice properties as in the smooth case, cf. ~\cite{BVY, Larkang}. It would be interesting to investigate the general situation more carefully.

The organization of this paper is as follows. In Sections ~\ref{rescurr} and ~\ref{toric} we provide some necessary background on residue currents and toric varieties, respectively. In Section ~\ref{arkesektion} we present a basic result, which essentially is a toric interpretation of Theorem 2.3 in ~\cite{A3}. Based on this we prove Theorems ~\ref{mac}-\ref{kollarsats} in Section ~\ref{results}, in which we also provide slightly more general formulations and consider the special case when $\Q$ is of the form \eqref{product}. Finally, in Section ~\ref{exs} we compare our results to previous work, interpret them in terms of usual degree bounds and give some examples. 

\textbf{Acknowledgment:} I would like to thank Mats Andersson, S{\'e}bastien Boucksom, Mircea Musta\c{t}\u{a}, Alexey Shchuplev, and Mart{\'i}n Sombra for fruitful discussions. Thanks to Maurice Rojas for pointing out the reference to the work of Castryck et al. Also thanks to the referee for careful reading and for many important remarks and helpful suggestions. This work was partially carried out when the author was visiting Institut Mittag-Leffler.

\section{Residue currents}\label{rescurr}
Let $f_1, \ldots, f_m$ be holomorphic functions whose common zero set $V_f$ has codimension $m$. Then the \emph{Coleff-Herrera product}, introduced in ~\cite{CH},
\begin{equation}\label{coleff}
R^f_{CH} = \dbar \left [\frac{1}{f_1}\right ]\wedge \cdots \wedge \dbar \left [\frac{1}{f_m}\right ],
\end{equation}
represents the ideal $(f)$ generated by the $f_j$ in the sense that it has support on $V_f$ and moreover if $\psi$ is a holomorphic function, then $\psi\in (f)$ locally if and only if $\psi R^f_{CH}=0$, see ~\cite{DS, P}.

Passare-Tsikh-Yger, \cite{PTY}, constructed residue currents by means of the Bochner-Martinelli kernel that generalize the Coleff-Herrera product to when the codimension of $V_f$ is arbitrary. Their construction was later developed by Andersson \cite{A}. We will use his global construction. 

\begin{thm}[Andersson  ~\cite{A}]
Let $f$ be a holomorphic section of a Hermitian vector bundle $E$ of rank $m$ over a complex manifold $X$ of dimension $n$. Then one can construct a $(\Lambda(E^*)$-valued) residue current $R^f$ on $X$, which has support on the zero locus $V_f$ of $f$ and satisfies: 
\begin{enumerate}
\item[(a)]
If $\psi$ is holomorphic on $X$ and 
$\psi R^f =0$,  
then $\psi$ is locally in the ideal $(f)$ generated by $f$. 
\item[(b)]
If $\codim V_f =m$ then $R^f$ is locally equal to a Coleff-Herrera product \eqref{coleff}; in particular, $\psi R^f=0$ if and only if $\psi\in (f)$ locally.
\item[(c)]
If $\psi$ locally satisfies
\begin{equation}\label{storlek}
|\psi|\leq C |f|^{\min (m,n)}
\end{equation}
for some constant $C$, then $\psi R^f=0$.
\end{enumerate}
\end{thm}

If $\psi$ is a holomorphic section of a line bundle $L$ over $X$, then $\psi\in (f)$ if there is a $g\in\O(X, E^*\otimes L)$ such that
\begin{equation}\label{delta}
\delta_f g=\psi,
\end{equation}
where $\delta_f$ is contraction (interior multiplication) with $f$. 
If  $\varepsilon_1, \ldots, \varepsilon_n$ is a local holomorphic frame for $E$ and $\varepsilon^*_1, \ldots, \varepsilon^*_n$ is the dual frame, so that $f=\sum_{i=1}^m f_i \varepsilon_i$ and $g=\sum g_i \varepsilon^*_i$, then \eqref{delta} just reads $\sum f_i g_i  =\psi$, that is, \eqref{mja}. 
Andersson's construction of $R^f$ is based on the Koszul complex, which, combined with solving $\dbar$-equations, is a classical tool for solving division problems, see for example ~\cite{Hormander}. Vaguely speaking, $R^f$ appears as an obstruction when one tries to extend a solution $g$ to the division problem \eqref{delta} from $X\setminus V_f$ to $X$. 
Let $s$ be the section of $E^*$ with pointwise minimal norm, such that $\delta_f s=|f|^2$, where $|\cdot |$ is the Hermitian metric on $X$, and let 
\begin{equation}\label{udef}
u=\sum_k \frac{s\wedge (\dbar s)^{k-1}}{|f|^{2k}}.
\end{equation}
Then $u$ is a section of $\Lambda (E^*\oplus T^*_{0,1}(X))$, which is clearly well-defined and smooth outside $V_f$ and moreover 
$\dbar |f|^{2\lambda} \wedge u$ 
has an analytic continuation as a current to where $\Re\lambda > -\varepsilon$. The current $R^f$ is defined as the value at $\lambda=0$. Locally the coefficients of $R^f$ are the residue currents introduced by Passare-Tsikh-Yger ~\cite{PTY}.

Morally, the residue current $R^f$ is an obstruction to solve \eqref{delta} locally on $X$. To glue these local solutions together to a global solution we need to solve certain $\dbar$-equations on $X$. 
The following result is a special case of Theorem 2.3 in ~\cite{A3}. 
\begin{thm}\label{abas}
Let $L$ be a line bundle over $X$. Assume that 
\begin{equation}\label{hos}
H^{0,q}(X, \Lambda ^{q+1} E^* \otimes L)=0
\end{equation}
for $1\leq q \leq \min (m-1, n)$. Let $\psi$ be a holomorphic section of $L$. 
If $\psi R^f=0$, then there is a $g\in\O(X, E^*\otimes L)$ that satisfies \eqref{delta}. 
\end{thm}

Given a holomorphic function $g$ we will use the notation $\dbar [1/g]$ for the value at $\lambda=0$ of $\dbar|g|^{2\lambda}/g$ and analogously by $[1/g]$ we will mean $|g|^{2\lambda}/g|_{\lambda=0}$. For further reference note that $g \dbar [1/g]=0$. 

The residue currents that appear in this paper allow for multiplication with characteristic functions of varieties, and more generally constructible sets, in such a way that ordinary calculus rules hold; in fact, they are \emph{pseudomeromorphic currents} in the sense of ~\cite{AW2}. In particular, if $R$ is a residue current on $X$ and $V\subset X$ is a variety, then $\psi R=0$ if and only if $\psi\1_V R =0$ and $\psi \1_{X\setminus V} R=0$. Also, if $\pi: X\to Y$ is a holomorphic modification and $W$ is a subvariety of $Y$, then 
\begin{equation}\label{uppner}
\1_W(\pi_* R)=\pi_*(\1_{\pi^{-1}(W)}R).
\end{equation}
Is $Z$ is an analytic variety with support $|Z|$, we will write $\1_Z$ for $\1_{|Z|}$. 

A pseudomeromorphic current with support on a variety $Z$ is said to have the \emph{Standard Extension Property (SEP)} (with respect to $Z$) in the sense of Bj\"ork ~\cite{Bj} if $\1_W T=0$ for all subvarieties $W \subset Z$ of positive codimension. The Coleff-Herrera product \eqref{coleff} has the SEP (with respect to $V_f$); in particular, $\dbar[1/g]$ has the SEP. 

One can define pseudomeromorphic currents also on singular varieties, so that the properties above hold true, see \cite{Larkang}.

\section{Toric varieties}\label{toric}
A toric variety is a partial compactification of the torus $T=(\C^*)^n$, which admits an action of $T$ that extends the action of $T$ on itself; for a general reference on toric varieties, see \cite{F}. A toric variety can be constructed from a \emph{fan} $\Delta$, which is a certain collection of lattice cones, by gluing together copies of $\C^n$ corresponding to $n$-dimensional cones of $\Delta$; we denote the resulting toric variety by $X_\Delta$.  Throughout this paper we will assume that the lattice is $\Z^n$. 
We will also assume that all fans $\Delta$ are \emph{complete}, that is, $\bigcup_{\sigma\in\Delta}\sigma=\R^n$; then the corresponding toric varieties are compact.

\subsection{Toric varieties from polytopes}\label{toricpoly}
Let $\poly$ be a lattice polytope in $\R^n$. Note that if $F$ is a polynomial, then $\supp F\subseteq\Q\subseteq\R_{\geq 0}^n$. Therefore we will assume that all lattice polytopes in this paper contained in $\R^n_{\geq 0}$. Let $\rho_1, \ldots, \rho_s$ be the normal vectors of the facets (faces of maximal dimension) of $\poly$, chosen in such a way that each $\rho_j$ is the shortest inwards pointing normal vector that has integer coefficients. Then $\poly$ admits a representation
\begin{equation}\label{rational}
\poly=\bigcap_j \{x\in\R^n \text{ such that } \langle x, \rho_j \rangle \geq -a_j\}
\end{equation} 
for some integers $a_j$. The polytope $\poly$ determines a complete fan $\Delta_\poly$ whose cones correspond to the faces of $\poly$; given a face $A$ of $\poly$, the corresponding cone $\sigma_A$ is generated by the $\rho_j$ for which $A$ is a face of the facet determined by $\rho_j$.

A toric variety $X_\Delta$ is smooth if and only if each cone in $\Delta$ is generated by a part of a basis for the lattice $\Z^n$, see ~\cite[p.~29]{F}. Such a fan is said to be \emph{regular}. A polytope $\poly$ is smooth precisely when $\Delta_\poly$ is regular, cf. the introduction. For each fan $\Delta$ there exists a refinement $\widetilde \Delta$ of $\Delta$ such that $X_{\widetilde \Delta} \to X_\Delta$ is a resolution of singularities, see ~\cite[p.~48]{F}. Also if $\Delta_1$ and $\Delta_2$ are two different fans, there exists a regular fan $\widetilde\Delta$ that refines both $\Delta_1$ and $\Delta_2$. If $\Delta$ is a refinement of $\Delta_\poly$ we say that $\Delta$ is \emph{compatible} with $\poly$, see ~\cite[p.~73]{F}.

\subsection{Divisors and line bundles}\label{dochl}
Each one-dimensional cone $\R_+\rho_j$ of a fan $\Delta$ determines a divisor $D_j$ on $X_\Delta$ that is invariant under the action of $T$. Moreover, any divisor on $X_\Delta$ is rationally equivalent to a $T$-invariant divisor, or $T$-divisor for short, so the $D_j$ generates the Chow group $A_{n-1}(X_\Delta)$ of Weil divisors modulo rational equivalence. 

A $T$-Cartier divisor on $X_\Delta$ is of the form 
$\sum_j \langle a, \rho_j \rangle D_j$, for some $a\in\Z^n$; we identify Cartier divisors with the corresponding Weil divisors. A $T$-Cartier divisor on $X_\Delta$ gives rise to a polytope $\poly_D$, compatible with $\Delta$. If  $D=\sum b_j D_j$, then $\poly_D=\bigcap_j \{x\in \R^n \text{ such that } \langle x, \rho_j \rangle \geq - b_j\}$. 
The global holomorphic sections of the line bundle $\O(D)$ correspond precisely to  polynomials with support in $\poly_D$.

A $T$-Cartier divisor $D$ also gives rise to a continuous piecewise linear function $\Psi_D$ on $\R^n$; if $D=\sum b_j D_j$, then $\Psi_D$ is defined by $\Psi_D(\rho_j)=-b_j$. In particular, $\Psi_D$ is linear on each cone of $\Delta$. The function $\Psi_D$ is said to be \emph{strictly concave} if it is concave and the linear functions defining it are different for different $n$-dimensional cones of $\Delta$. Concavity of $\Psi_D$ is related to positivity of the line bundle $\O(D)$: $\O(D)$ is generated by its sections if and only if $\Psi_D$ is concave and it is ample if and only if $\Psi_D$ is strictly concave. 
It follows that the line bundle $\O(D_\poly)$ is ample on $X_\poly$. Moreover, if $\Delta$ is compatible with $\poly$ of the form \eqref{rational}, then $\poly$ determines a $T$-Cartier divisor $D_\poly= \sum a_j D_j$ on $X_{\Delta}$ such that $\poly_{D_\poly}=\poly$ and the line bundle $\O(D_\poly)$ is generated by its sections.

\subsection{Line bundle cohomology}\label{cohomology}
If $\Delta$ is complete and $L$ is a line bundle over $X_\Delta$, which is generated by its sections, then $H^{0,q}(X_\Delta,L)=0$ for all $q\geq 1$. By Serre duality, 
$H^{0,q}(X, -L)= H^{0,n-q}(X, L+K_X)$,  
where $K_X$ denotes the \emph{canonical divisor} on $X$. The canonical divisor on $X_\Delta$ is given as $K_X=-\sum D_j$, where $D_j$ are the irreducible divisors corresponding to the one-dimensional cones of $\Delta$. We conclude the following.  
\begin{lma}\label{today}
If $\Delta$ is compatible with $\poly$, then $H^{0,q}(X_\Delta, \O(D_{c\poly}))=0$ for all $c\geq 0$, for which $c\Q$ is a lattice polytope, and $q\geq 1$.

If moreover $\O(D_\poly +K_{X_\Delta})$ is generated by its sections, then $H^{0,q}(X_\Delta, \O(-D_{c\poly}))=0$ for $1\leq q\leq n-1$ and any $c\geq 1$, for which $c\Q$ is a lattice polytope.
\end{lma}
To see the second statement, note that for $c\geq 1$, $\Psi_{D_{c\poly}+K_{X_\Delta}}$ is concave as soon as $\Psi_{D_\poly+K_{X_\Delta}}$ is.

Let $\O(a)$ denote the line bundle over $\P^n$ whose sections correspond to $a$-homogeneous polynomials. Recall the following well known vanishing theorem, see for example ~\cite[Thm.~10.7, p.~437]{Dem}.
\begin{thm}\label{forsvinnande}
It holds that $H^{0,q}(\P^n, \O(a))=$ if (and only if)
$q=0$ and $a<0$, $1\leq q\leq n-1$, or $q=n$ and $a \geq -n$.
\end{thm}

Given line bundles $L_1\to X_1$ and $L_2\to X_2$, let $L_1\boxtimes L_2\to X_1\times X_2$ denote the tensor product of the pullbacks of $L_1$ and $L_2$ to $X_1\times X_2$. 
By the \emph{K\"unneth Formula}, we have:
\begin{multline}\label{kunneth}
H^{0,q}\big (\P^{n_1}\times\cdots\times\P^{n_r}, \O(a_1)\boxtimes\cdots\boxtimes\O(a_r) \big )= \\
\bigoplus_{q_1+\cdots + q_r=q} H^{0,q_1}\big (\P^{n_1}, \O(a_1)\big )\otimes\cdots\otimes H^{0,q_r}\big (\P^{n_r}, \O(a_r)\big ).
\end{multline}

\begin{ex}\label{pluttex}
Assume that $\Q$ is a product of simplices, that is, $\Q$ is of the the form \eqref{product}. Set $n:=n_1+\cdots + n_r$. Then $\Q$ has normal directions 
$\rho_1=e_1, \ldots, \rho_n=e_n, \rho_{n+1},\ldots, \rho_{n+r}$, where $\rho_{n+1}$ has $-1$ in the first $n_1$ positions and zeros elsewhere, and for $2\leq k \leq r$, $\rho_{n+k}$ has $-1$ in position $n_1+\cdots + n_{k-1}+1, \ldots, n_1 + \ldots + n_k$ and zeros elsewhere. 
The fan $\Delta_\Q$ is regular so that $\Q$ is smooth. In fact, $X_\Q=\P^{n_1}\times\cdots\times\P^{n_r}$.

Now $D_\Q= \sum d_j D_{n+j}$ and the line bundle $\O(D_\Q)$ is just the line bundle $\O(d_{1})\boxtimes\cdots\boxtimes\O(d_{r})$  over $\P^{n_1}\times\cdots\times\P^{n_r}$. 
Note that $\Psi_{D_\Q}$ is concave, which means that $\O(D_\Q)$ is generated by its sections, precisely when 
$d_k\geq 0$ for $1\leq k\leq r$ and that $\Psi_{D_\Q}$ is strictly concave, which means that $\O(D_\Q)$ is ample, precisely when 
$d_k \geq 1$ for $1\leq k\leq r$.

We claim that $H^{0,q}(X_\Q, \O(D_{c\Q}))=0$ if $1\leq q \leq n-1$ and $c$ is any integer. More precisely, if $1\leq q\leq n-1$, then \eqref{kunneth} vanishes as soon as either all $a_i \geq 0$ or all $a_i < 0$. To see this note that if $q>0$ then each term in the left hand side of \eqref{kunneth} has at least one factor $H^{0,q_j}(\P^{n_j}, \O(a_j))$ for which $q_j>0$. Now if $a_i \geq - n_i$ or $a_i < 0$ for all $i$, this factor vanishes according to Theorem ~\ref{forsvinnande}. Similarly, if $q<n$, then each terms has a factor for which $q_j<n_j$ and so this factor vanishes if $a_j <0$, which proves the claim.
\end{ex} 

\subsection{Homogeneous coordinates on toric varieties}\label{homo}
The homogeneous coordinate ring $S$ on a toric variety $X_\Delta$ was introduced by Cox ~\cite{Cox} as a generalization of homogeneous coordinates on projective space. 
The ring $S$ has one variable $z_j$ for each one-dimensional cone $\R_+ \rho_j$ in the fan $\Delta$ or, equivalently, for each irreducible $T$-Weil divisor $D_j$ on $X_\Delta$. Moreover $S$ has a grading inherited from the Chow group $A_{n-1}(X_\Delta)$: the degree of a monomial $\prod z_j^{a_j}$ is $[\sum a_j D_j]\in A_{n-1}(X_\Delta)$. Let $D=\sum a_j D_j$ be a $T$-divisor on $X_\Delta$. The global sections of the line bundle $\O(D)$ can then be expressed as polynomials in the monomials $\mu_b=\prod_j z_j^{\langle b, \rho_j \rangle + a_j}$, where $b=(b_1, \ldots, b_n)\in \poly_D \cap \Z^n$. If $X_\Delta$ is smooth, then local coordinates in the affine chart $\U_\sigma$ corresponding to the $n$-dimensional cone $\sigma$ is obtained by setting $z_j=1$ if $\R_+ \rho_j$ is not a facet of $\sigma$, see, for example,  ~\cite{Weimann}.

In this paper we want to consider toric varieties that are compactifications of $\C^n$. 
Assume that $\poly$ is a lattice polytope that contains the origin and the lattice points $e_1, \ldots, e_n$, that is, the support of the coordinate functions $z_1,\ldots, z_n$. Given such a polytope $\poly$ one can always find a regular fan, compatible with $\poly$, that contains the $n$-dimensional cone $\sigma_0$ generated by $\rho_1=e_1, \ldots, \rho_n=e_n$; in fact, $\sigma_0$ is the first orthant in $\R^n$. Let $\Delta$ be such a fan. Then, in the representation \eqref{rational} of $\poly$,  $a_1=\ldots =a_n=0$. It follows that 
\begin{equation*}
\mu_b=z_1^{b_1}\cdots z_n^{b_n} z_{n+1}^{\langle b, \rho_{n+1}\rangle + a_{n+1}} \cdots z_{n+s}^{\langle b, \rho_{n+s}\rangle + a_{n+s}}.
\end{equation*}
Thus, in local coordinates in $\U_{\sigma_0}$,  $\mu_b=z_1^{b_1}\cdots z_n^{b_n}=z^b$, and so $\mu_b$ can really be thought of as a homogenization of the monomial $z^b$; we will refer to a global section of $\O(D_\poly)$ as the \emph{$\poly$-homogenization} of the corresponding polynomial in $\U_{\sigma_0}$. We will identify the chart $\U_{\sigma_0}$ with our original $\C^n$ and refer to $X_{\Delta}\setminus \U_{\sigma_0}=\bigcup_{j \geq n+1} D_j$ as \emph{the variety at infinity} and denote it by $V_\infty$. 

Let us remark that by working on toric varieties obtained from arbitrary polytopes we could probably obtain results for Laurent polynomial in $(\C^*)^n$, cf. ~\cite[Theorem~2]{Sombra}.

\section{The basic result}\label{arkesektion}
The following basic result is a consequence of Theorem \ref{abas}.

\begin{thm}\label{arkesatsen}
Let $F_1,\ldots, F_m$, and $\Psi$ be polynomials in $\C [z_1,\ldots, z_n]$, and let  $\Q_j\supseteq \supp F_j$ and $\Qq\supseteq \supp \Psi$ be lattice polytopes that  contain the origin in $\R^n$. Assume that $\Delta$ is a regular fan, compatible with $\Q_j$ and $\Qq$, that contains the first orthant in $\R^n$ as a cone, and that 
\begin{equation}\label{arke}
H^{0,q}\Big (X_\Delta, \O \big (D_\Qq -( D_{\Q_{j_1}}+ \cdots + D_{\Q_{j_{q+1}}})\big)\Big) =0 
\end{equation}
for $1\leq q\leq \min (m-1,n)$ and all $\J=\{j_1,\ldots,j_{q+1}\}\subseteq \{1,\ldots,m\}$. 
Assume moreover that 
\begin{equation}\label{residy}
\psi R^f=0,
\end{equation}
where $\psi$ is the $\Qq$-homogenization of $\Psi$ and $f$ is the section $(f_1,\ldots, f_m)$ of $\O(D_{\Q_1})\oplus\cdots\oplus\O(D_{\Q_m})$ over $X_\Delta$, where $f_j$ is the $\Q_j$-homogenizations of $F_j$. 

Then there are polynomials $G_1,\ldots, G_m$ such that 
\begin{equation}\label{division}
\sum_{j=1}^m F_j G_j = \Psi
\end{equation}
and
\begin{equation}\label{uppskattning}
\supp F_j G_j \subseteq  \Qq. 
\end{equation}
\end{thm}

In general, \eqref{arke} is satisfied if $\O(D_\Qq)$ is positive enough. For example, if $D_\Q$ is ample, then there is an $r$ such that \eqref{arke} holds for $\Qq=s\Q$ if $s\geq r$. 

If $\Q_j=d_j \Sigma^n$, where $d_j=\deg F_j$, and $\Qq$ is of the form $c\Sigma^n$, we can choose $X$ as $\P^n$. Then 
$\O(D_\Qq-(D_{\Q_{j_1}}+\cdots + D_{\Q_{j_{q+1}}}))$ is the bundle $\O(c-d_{j_1}-\cdots - d_{j_{q+1}})$ over $\P^n$, 
and so by Theorem ~\ref{forsvinnande}, \eqref{arke} is satisfied if $m\leq n$ or $c\geq d_1+\cdots + d_{n+1}-n$ if the $d_j$ are ordered so that $d_1\geq \ldots \geq d_m$; this is Theorem 1.1 in ~\cite{A3}. 
In this paper we generalize this basic situation in two directions: we consider the case when $\Q_j$ of the form $d_j \Q$, where $\Q$ is a fixed polytope (with certain properties), and the case when $\Q_j$ is a product of simplices.

Let $\a$ denote the ideal sheaf over $X$ generated by the tuple $f_1,\ldots, f_m$, let $\pi:X^+ \to X$ be the normalization of the blow-up of $\a$, and let $[D]=\sum r_i [D_i]$ be the associated divisor in $X^+$. Then $\psi$ is in the integral closure $\overline\a$ of $\a$ if $\pi^* \psi$ vanishes at least to order $r_j$ on each divisor $D_j$, see for example ~\cite{Laz}. In particular, if we let $r:= \max_j r_j$, then \eqref{storlek} is satisfied if $\psi$ vanishes to order $\min(m,n) r$ along $V_f$. Recall from Section ~\ref{rescurr} that \eqref{residy} is satisfied if and only if $\psi\1_{\C^n}R^f=0$ and $\psi\1_{V_\infty}R^f=0$.

\begin{lma}\label{oandligheten}
Assume that $\psi$ vanishes to order $\min(m,n) r$ along $V_\infty$. Then $\psi \1_{V_\infty}R^f=0$. 
\end{lma}

\begin{proof}
One can show that $\pi^*(\dbar |f|^{2\lambda}\wedge u)$, where $u$ is defined by \eqref{udef}, has an analytic continuation as a current on $X^+$ to where $\Re\lambda >-\epsilon$, such that $\pi_*R^+=R^f$, where $R^+=\pi^*(\dbar |f|^{2\lambda}\wedge u)|_{\lambda=0}$, see \cite{Larkang}.

In $X^+$, $\pi^* f=f_0 f'$, where $f_0$ is holomorphic and $f'$ is a nonvanishing tuple. It follows that $R^+$ is of the form $\sum_{k=1}^{\min(m,n)}\dbar[1/f_0^k]\wedge \alpha_k$, where $\alpha_k$ are smooth, cf. \cite[Pf of Thm 1.1]{A}. Since $\dbar[1/f_0^k]$ has the SEP with respect to (the support of) $D$, so has $R^+$. It follows that  
$R^+=\sum_{D_j\subseteq D} \1_{D_j} R^+$ and moreover, using \eqref{uppner}, 
$\1_{V_\infty}R^f=\sum_{\pi(D_j)\subseteq V_\infty} \pi_*(\1_{D_j} R^+)$. 
Let $Z$ denote the union of the singular locus of $X^+$ and the singular locus of $D$. Then $Z$ has codimension 2 in $X^+$ and so $R^+=\1_{X^+\setminus Z}R^+$.

Assume that $\psi$ vanishes to order $\min(m,n) r$ along $V_\infty$. We need to show that $(\pi^*\psi) \1_{D_j\setminus Z} R^+=0$ if $D_j$ is one of the divisors that are mapped into $V_\infty$. Let $D_j$ be such a divisor. Then locally on $D_j\setminus Z$, $f_0=\sigma^{r_j}$, where $\sigma$ is a local defining function for $D_j$. Moroever $\pi^*\psi$ is divisible by $\sigma^{\min(m,n) r}$ and consequently it annihilates $\1_{D_j\setminus Z}R^+ = \sum_k \dbar[1/\sigma^{k r_j}] \wedge \alpha_k$.  Hence $\psi \1_{V_\infty}R^f=0$.
\end{proof}

\begin{remark}\label{rremark}
In some cases we can estimate $r$. Let us follow ~\cite[Chapter~10.5]{Laz}. Suppose that $D$ is a divisor on $X$ such that $\O_X(D)\otimes\a$ is globally generated. Then Proposition 10.5.5 in ~\cite{Laz} asserts that 
\begin{equation*} 
\sum_j r_j \cdot \deg_D(Z_j) \leq \deg_D(X)
\end{equation*}
where $Z_j=\pi(D_j)$ are the so-called \emph{distinguished varieties} associated with $\a$. If moreover $D$ is ample, then $\deg_D(Z_i)>0$ and so we get the following rough estimate of $r$:
\begin{equation*}
r\leq \deg_D(X)
\end{equation*}

Let $\Q$ be a smooth polytope that contains the supports of the $F_j$. Then $\O_{X_\Q}(D_\Q)\otimes \a$ is globally generated and $D_\Q$ is an ample divisor on $X_\Q$. Moreover $\deg_{D_\Q}(X_\Q) = \vol (\Q)$, see ~\cite[Prop.~2.10]{Oda}. Thus $r\leq \vol (\Q)$.

\end{remark}

\begin{proof}[Proof of Theorem \ref{arkesatsen}]
Let $E$ be the bundle $\O(D_{\Q_1})\oplus \cdots \oplus \O(D_{\Q_m})$ over $X_\Delta$ and let $L=\O(D_\Qq)$. 
Then 
\[
\Lambda^q E^* \otimes L = 
\bigoplus_{\J=\{j_1,\ldots, j_q\}\subseteq \{1,\ldots, m\}}\O\big (D_\Qq-(D_{\Q_{j_1}}+\cdots + D_{\Q_{j_q}})\big ),
\]
and so \eqref{hos} holds for $1\leq q \leq \min(m-1,n)$ if \eqref{arke} holds for $1\leq q\leq \min(m-1,n)$ and any multi-index $\J$ of length $q+1$. 
Thus, if $\psi\in\O(X_\Delta,L)$ annihilates the residue current $R^f$, then Theorem 
~\ref{abas} asserts that we can find a $g=(g_1,\ldots, g_m) \in\O(X_\Delta,E^*\otimes L)$ that satisfies \eqref{delta}. 
Dehomogenizing gives polynomials $G_1, \ldots, G_m$ in $\C [z_1,\ldots, z_n]$ that satisfy \eqref{division} and \eqref{uppskattning}. 
\end{proof}

\section{Results and proofs}\label{results}
In this section we deduce Theorems \ref{mac}-\ref{kollarsats} from Theorem \ref{arkesatsen}. We provide slightly more general formulations of some of the results and we also give sharper estimates in the special case when $\Q$ is a product of simplices, which corresponds to separate degree bounds in subsets of the variables. From now on let us use the shorthand notation $\mu:=\min(m,n)$. Also throughout the paper $F_1, \ldots, F_m$, and $\Phi$ are assumed to be polynomials in $\C [z_1,\ldots, z_n]$.

\subsection{Sparse versions of Macaulay's Theorem} 
\label{macresults}
In Theorem ~\ref{mac} the $F_j$ are assumed to have no common zeros neither in $\C^n$ nor at infinity. This should be interpreted as that $\Q$ necessarily contains the origin and the $\Q$-homogenizations $f_j$ of the $F_j$ lack common zeros in $X_\Delta$ if $\Delta$ is compatible with $\Q$. Observe that, whether the $f_j$ have common zeros in $X_\Delta$ in fact only depends on $\Q$ and not on the particular choice of $\Delta$.

Theorem ~\ref{mac} is a direct consequence of the following more general result, which was proved  for polynomials over arbitray fields, or even DVRs, by Castryck-Denef-Vercauteren, ~\cite{CDV}. We include a proof for completeness.  Theorem ~\ref{mac} corresponds to $d_j=1$ and $\Phi=1$. Tuitman, ~\cite{Tuitman}, proved a generalization of Castryck-Denef-Vercauteren's result, in which he allows the polynomials to have support in different polytopes, see also \cite{W}.  

\begin{thm}\label{macsamma}
Assume that $F_j$ has support in the lattice polytope $d_j \Q$, where $\Q$ is a fixed lattice polytope that contains the origin and the $d_j$ are ordered so that $d_1 \geq \cdots\geq d_m$. Assume that the $F_j$ have no common zeros neither in $\C^n$ nor at infinity, meaning that the $d_j \Q$-homogenizations of the $F_j$ lack common zeros. Assume that $\Phi$ has support in the lattice polytope $e\Q$. Then there are polynomials $G_j$ that satisfy 
\eqref{ideal} 
and 
\begin{equation}\label{finlasse}
\supp (F_j G_j)\subseteq \max (\sum_{j=1}^{n+1} d_j, e) \Q.
\end{equation}
\end{thm}

\begin{proof}
Let $\Q_j=d_j\Q$, let $\Delta$ be regular and compatible with $\Q$. Since $\Q\subseteq\R^n_+$ contains the origin, we can choose $\Delta$ so that it contains the first orthant. Moreover, let $\Qq= c\Q$, where $c=\max (d_1+\cdots + d_{n+1},e)$. Then 
\[
\O\big (D_\Qq-(D_{\Q{j_1}}+\ldots + D_{\Q_{j_{q+1}}})\big )= \O\big (D_{(c-(d_{j_1}+\ldots + d_{j_{q+1}}))\Q}\big ),
\]
where $c-(d_{j_1}+\ldots + d_{j_{q+1}})\geq 0$ if $q\leq n$. It follows by Lemma ~\ref{today} that \eqref{arke} is satisfied for $1\leq q \leq \min(m-1,n)$ and any multi-index $\J$ of length $q+1$.

Let $f_j$ be the $\Q_j$-homogenizations of the $F_j$, let $R^f$ be the corresponding residue current, and let $\psi$ be the $\Qq$-homogenization of $\Phi$. 
Since the $f_j$ lack common zeros, $R^f=0$ and thus \eqref{residy} is trivially satisfied. Hence Theorem ~\ref{arkesatsen} asserts that there are polynomials $G_j$ that satisfy \eqref{finlasse}.
\end{proof}

The following result appeared in ~\cite[Theorems~10.2 and 13.4]{Ahl}. The proof given there uses Koszul complex methods. 
For completeness we give a proof using Theorem ~\ref{arkesatsen}.
\begin{thm}\label{macproduct}
Assume that $F_j$ has support in 
\[
\Q_j=d_{j1}\Sigma^{n_1}\times\cdots\times d_{jr}\Sigma^{n_r},
\]
where $n_1+\cdots + n_r=n$, and moreover that the $F_j$ have no common zeros neither in $\C^n$ nor at infinity in $\P^{n_1}\times\cdots\times\P^{n_r}$. 
Let $k_1, \ldots, k_r$ be a permutation of $1, \ldots, r$ and let 
\begin{equation}\label{cerat}
c_{k_\ell} = \max_{\J \text{ such that } |\J|=n_{k_\ell}+\cdots + n_{k_r}+1} 
\sum_{i =1}^{n_{k_\ell}+\cdots + n_{k_{r}}+1} d_{j_i k_\ell}-n_{k_\ell}.
\end{equation}
Then there are polynomials $G_j$ that satisfy \eqref{ett}
and
\begin{equation*}
\supp (F_j G_j)\subseteq c_1 \Sigma^{n_1}\times\cdots\times c_r\Sigma^{n_r}.
\end{equation*}
\end{thm}

The condition \eqref{cerat} means that $c_{k_\ell}$ is equal to the sum of the $n_{k_\ell} +\cdots + n_{k_r}+1$ largest $d_{j k_\ell}$ minus $n_{k_\ell}$. In particular, if $\Q_j=\Q$ of the form \eqref{product}, then $c_{k_\ell}=(n_{k_\ell}+\cdots + n_{k_r} +1)d_{k_\ell} - n_{k_\ell}$.

Macaulay's Theorem ~\cite{Mac} corresponds to the case when $\Q_j= d_j \Sigma^n$, where $d_j=\deg F_j$ and the $d_j$ are ordered so that $d_1\geq\ldots\geq d_m$:

\noindent
\emph{
Assume that $F_j$ have no common zeros even at infinity (in $\P^n$). Then one can find $G_j$ that satisfy \eqref{ett} and $\deg (F_j G_j) \leq \sum_{j=1}^{n+1}d_j-n$.}

\begin{proof}
Let $X=X_{\Q_j}=\P^{n_1}\times \cdots \times \P^{n_r}$, cf. Example ~\ref{pluttex}, and let $\Qq=c_1 \Sigma^{n_1}\times\cdots\times c_r \Sigma^{n_r}$.  Note that $\Delta_{\Q_j}$ contains the first orthant. Moreover note that 
\begin{multline}\label{trott}
H^{0,q}\Big (X, \O\big (D_\Qq-(D_{\Q_{j_1}}+\cdots + D_{\Q_{j_{q+1}}})\big )\Big )= \\
H^{0,q}\Big (\O\big (c_1 - \sum_{i=1}^{q+1} d_{j_i 1}\big )\boxtimes\cdots\boxtimes\O\big (c_r - \sum_{i=1}^{q+1} d_{j_i r}\big )\Big ). 
\end{multline}
By the K\"unneth formula, the right hand side of \eqref{trott} is equal to 
\begin{equation}\label{wetter}
\bigoplus_{q_1+\cdots +q_r=q} 
H^{0,q_1}\Big (\P^{n_{1}}, \O\big (c_{1}-\sum_{i=1}^{q+1} d_{j_i,1}\big )\Big ) 
\otimes\cdots\otimes 
H^{0,q_r}\Big (\P^{n_{r}}, \O\big (c_{r}-\sum_{i=1}^{q+1} d_{j_i,r}\big )\Big ).
\end{equation}
If $n_{k_2}+\ldots + n_{k_r} +1 = n-n_{k_1}+1 \leq q \leq n$, then $q_{k_1}\geq 1$ in all terms in \eqref{wetter}. Thus by \eqref{cerat} and Theorem ~\ref{forsvinnande} the factor 
\begin{equation}\label{q1}
H^{0,q_{k_1}}\big (\P^{n_{k_1}}, \O (c_{k_1}-\sum d_{j_i k_1})\big),
\end{equation}
in each term vanishes since the sum contains $q+1\leq n+1$ terms. 

If $n-n_{k_1}-n_{k_2}+1 \leq q \leq n-n_{k_1}$, then, in each term in \eqref{wetter}, either $q_{k_1}\geq 1$ or $q_{k_2}\geq 1$. In the first case \eqref{q1} vanishes as above. In the second case 
$H^{0,q_{k_2}}\big (\P^{n_{k_2}}, \O(c_{k_2}-\sum d_{j_i k_2})\big )$ vanishes. 

Hence \eqref{trott} vanishes for $n-n_{k_1}-n_{k_2}+1\leq q \leq n$. It follows by induction, using \eqref{cerat}, that \eqref{trott} vanishes for $1\leq q\leq n$ and any multi-index $\J=\{j_1,\ldots , j_{q+1}\}$.

Let $R^f$ be the residue current associated with the $\Q_j$-homogenizations of the $F_j$ and let $\psi$ be the $\Qq$-homogenization of $1$. As in the proof of Theorem ~\ref{macsamma}, \eqref{residy} is trivially satisfied and so Theorem ~\ref{arkesatsen} gives the desired polynomials $G_j$. 
\end{proof}

\subsection{Sparse versions of N\"other's $AF+ BG$ Theorem}\label{maxresults}
The assumption in Theorem \ref{max} that the zero set of the $F_j$ has no component contained in the variety at infinity should be interpreted as that for some $d_j$, such that $d_j\Q$ are lattice polytopes, the zero set $V_f$ of the $d_j \Q$-homogenizations of $F_j$ has no irreducible component contained in $V_\infty$ in $X_\Q$. Note that whether $V_f$ has a component contained in $V_\infty$ in $X_\Delta$, where $\Delta$ is compatible with $\Q$, actually does depend on $\Delta$. Indeed, in general $V_f$ blows up as $\Delta$ is refined. 

Moreover by $\Q$ being large we mean that $\O(D_\Q+K_{X_\Q})$ is generated by its sections. Roughly speaking this is satisfied if the faces of $\Q$ are large enough. In particular, given a smooth polytope $\Q$ then for some large enough integer $b$ the polytope $b\Q$ is large. In fact, Fujita's conjecture, which holds for toric varieties,  asserts that $b\Q$ is large if $b\geq n+1$, see ~\cite{Payne}. The assumption that $\Q$ is large and smooth is used in the proof of Theorem \ref{max}; we do not know whether they are necessary for the validity of the theorem. 

Let us give a more precise formulation of Theorem ~\ref{max}. Let $\N$ denote the natural numbers $1,2,\ldots$. 
\begin{thm}\label{maxforfinad}
Let $\Q$ be a smooth polytope, that contains the origin and the support of the coordinate functions $z_1,\ldots, z_n$ and that satisfies that the line bundle $\O(D_\Q+K_{X_\Q})$ over $X_\Q$ is generated by its sections. Assume that for some $d_j\in\N$, the zero set $V_f$ of the $d_j \Q$-homogenizations of the $F_j$ has codimension $m$ and moreover $V_f$ has no irreducible component contained in $V_\infty$ in $X_\Q$.

Assume that $\Phi\in (F_1,\ldots, F_m)$ and that $\supp \Phi \subseteq e \Q$, where $e\in\N$. Then there are polynomials $G_j$ that satisfy \eqref{ideal} and 
\begin{equation*}
\supp (F_j G_j) \subseteq e\Q.
\end{equation*}
\end{thm}

\begin{proof}
Let $\Q_j=d_j \Q$ and $\Delta=\Delta_\Q$. Then $\Delta$ contains the first orthant and $X_\Delta=X_\Q$ is smooth. Note that the fact that the codimension of the zero set of the $F_j$ is $m$ implies that $m\leq n$. By the second part of Lemma ~\ref{today}, for $1\leq q \leq \min(m-1,n)\leq n-1$, \eqref{arke} is thus satisfied for any polytope $\Qq$ of the form $\Qq=c\Q$, where $c\in\Z$, in particular for $\Qq=e\Q$. 

Let $R^f$ be the residue current associated with the $\Q_j$-homogenizations of the $F_j$ and let $\psi$ be the $\Qq$-homogenization of $\Phi$. Since $\codim V_f =m$ Theorem ~\ref{residysats} implies that $R^f$ is locally a Coleff-Herrera product. It follows that $\psi \1_{\C^n} R^f=0$ since $\Phi\in (F)$. Moreover $\1_{V_\infty}R^f=0$,  since $V_f$ has no component contained in $V_\infty$ and $R^f$ has the SEP, see Section \ref{rescurr}. 
Hence \eqref{residy} is satisfied and so Theorem ~\ref{arkesatsen} gives the result. 
\end{proof}

\begin{remark}
In light of (the last part of) Example \ref{pluttex}, Theorem \ref{maxforfinad} holds true also if $\Q$ is a product of simplices, that is, if $\Q$ of the form \eqref{product}, even if $\O(D_\Q + K_{X_\Q})$ is not generated by its sections. 
\end{remark}

\subsection{Sparse versions Andersson-G\"otmark's and Hickel's Theorems and the Nullstellensatz}\label{inftyresults}
In general, to satisfy \eqref{residy}, $\psi$ has to annihilate $R^f$ both in $\C^n$ and at infinity. In the above situations the latter condition was trivially satisfied.

The assumption that $\Q$ is smooth is used in the proofs of Theorems ~\ref{agsats}-\ref{kollarsats}; we do not know if it is necessary for the validity of the results.

\begin{remark}\label{kvall}
Let $\Q$ be a lattice polytope and $\Delta$ a regular fan compatible with $\Q$. Assume that $\Delta$ contains the first orthant, generated by $\rho_1=e_1, \ldots, \rho_n=e_n$, and that $D_\Q=\sum_{j=n+1}^{n+r} a_j D_j$ on $X_\Delta$. Recall from Section ~\ref{homo} that the $\Q$-homogenization $\tilde 1$ of $1$ is given by $\tilde 1 = \prod_{j=n+1}^{n+r} z_j^{a_j}$. Note that $\tilde 1$ vanishes to order $a_\infty := \min_{j\geq n+1} a_j$ along $V_\infty=\bigcup_{j=n+1}^{n+r}D_j$. If $\Q$ is of the form \eqref{product} then $a_\infty= \min_j d_j$. Note that $a_\infty$ is bounded from below by the minimal side length of ~$\Q$.
\end{remark}

\begin{proof}[Proof of Theorem ~\ref{agsats}]
Let $\Delta=\Delta_{\Q}$. Then $\Delta$ contains the first orthant and $X_\Delta=X_\Q$ is smooth. Moreoever,  let $\Q_j=\Q$ and $\Qq=c\Q$, where $c=\lceil e + m \vol(\Q)/a \rceil$. 
Then 
\[
\O\big (D_\Qq-(D_{\Q{j_1}}+\ldots + D_{\Q_{j_{q+1}}})\big )= \O\big (D_{(c-(q+1))\Q}\big ).
\]
Note that $\lceil e + m \vol(\Q)/a \rceil \geq m$; indeed, $\vol (\Q)/a \geq 1$. 
It follows from Lemma ~\ref{today} that \eqref{arke} is satisfied for $1\leq q\leq \min (m-1,n)$ and any ~$\J$ of length $q+1$.

Let $R^f$ be the residue current associated with the $\Q$-homogenizations of the $F_j$ and let $\psi$ be the $\Qq$-homogenization of $\Phi$. By Theorem ~\ref{residysats}(b) the assumption that $\Phi\in (F_1,\ldots, F_m)$ implies that $\psi$ annihilates $R^f$ in $\C^n$, that is, $\psi\1_{\C^n} R^f=0$.  Moreover, according to Remark ~\ref{kvall}, $\psi$ vanishes to order $\geq m \vol (\Q)$ along $V_\infty$, which by Lemma ~\ref{oandligheten} and Remark ~\ref{rremark} means that $\psi\1_{V_\infty}R^f=0$. Thus $\psi$ satisfies \eqref{residy} and now the result follows from Theorem ~\ref{arkesatsen}.
\end{proof}

\begin{proof}[Proof of Theorem ~\ref{hickelsats}]
Let $\Q_j=\Q$, let $\Delta=\Delta_\Q$, and let $\Qq=c\Q$, where $c=\max (\lceil \mu (e+ \vol (\Q) /a)\rceil, \min (m,n+1))$. 
Clearly $c\geq \min (m,n+1)$. It follows from Lemma ~\ref{today} that \eqref{arke} is satisfied for the required $q$ and $\J$; cf. the proof of Theorem \ref{agsats}. 

Let $R^f$ be the residue current associated with the $\Q$-homogenizations of the $F_j$ and let $\psi$ be the $\Qq$-homogenization of $\Phi^\mu$. Then, by Theorem ~\ref{residysats}(c), $\psi \1_{\C^n}R^f=0$ since $\Phi\in\overline{(F)}$. Moreover, in light of Remark ~\ref{kvall}, $\psi$ vanishes at least to order $\mu \vol (\Q)$ along $V_\infty$, which by Lemma ~\ref{oandligheten} and Remark ~\ref{rremark} implies that $\psi \1_{V_\infty} R^f=0$. Thus $\psi$ satisfies \eqref{residy} and Theorem ~\ref{arkesatsen} gives the result.
\end{proof}

\begin{proof}[Proof of Theorem ~\ref{kollarsats}]
Let $\Q_j=\Q$, let $\Delta=\Delta_\Q$, and let $\Qq=c\Q$, where 
$c=\max (\lceil \mu \vol (\Q) (1/a+e)\rceil, \min (m,n+1))$. 
It follows from Lemma ~\ref{today} that \eqref{arke} is satisfied for the required $q$ and $\J$. 

Let $R^f$ be the residue current associated with the $\Q$-homogenizations of the $F_j$  and let $\psi$ be the $\Qq$-homogenization of $\Phi^{\mu\vol(\Q)}$. 
Then, in light of Remark ~\ref{kvall}, $\psi$ vanishes at least to order $\mu \vol (\Q)$ at the zero set of the $f_j$ including $V_\infty$, which by Theorem ~\ref{residysats}(c) and the discussion after Theorem ~\ref{arkesatsen} implies that $\psi R^f=0$. Thus $\psi$ satisfies \eqref{residy} and Theorem ~\ref{arkesatsen} gives the result.
\end{proof}

\begin{remark}\label{dagas}                                           
In light of the above proofs, note that, in the formulations of Theorems ~\ref{agsats}-\ref{kollarsats}, as well as Theorems ~\ref{agprod}-\ref{kollarprod} below, we could in fact replace $\vol (\Q)$ by the order $r$ of vanishing at infinity, as defined in Section ~\ref{arkesektion}. This would allow us to drop the assumption that $\Q$ is smooth. However, we only know how to estimate $r$ when $\Q$ is smooth, and then by the rather rough estimate $r\leq \vol (\Q)$. In many cases one can do much better.

Moreover we could replace the minimal side length $a$ of $\Q$ by $a_\infty$, as defined in Remark ~\ref{kvall}, and  the polytopes of the form $\lceil c \rceil \Q$ could be replaced by the smallest lattice polytopes that contains $c \Q$. 
\end{remark}

If $\Q$ is a product of lattice polytopes one can get somewhat sharper estimates. For $\Q$ of the form \eqref{product} we get the following versions of Andersson-G\"otmark's and Hickel's Theorems and the Nullstellensatz.

\begin{thm}\label{agprod}
Assume that $F_j$ has support in $\Q$ of the form \eqref{product}. Moreover, assume that the codimension of the common zero set of $F_1, \ldots, F_m$ in $\C^n$ is $m$, that $\Phi\in(F_1, \ldots, F_m)$, and that 
\begin{equation}\label{phistod}
\supp \Phi \subseteq e_1 \Sigma^{n_1} \times \cdots \times e_r \Sigma^{n_r}.
\end{equation}
Then there are polynomials $G_j$ that satisfy \eqref{ideal} and  
\begin{equation*}
\supp (F_j G_j) \subseteq 
\prod_{j=1}^r
\big (e_j + m \vol(\Q)\big ) \Sigma^{n_j}.
\end{equation*}
\end{thm}

\begin{thm}\label{hickelprod}
Assume that $F_j$ has support in $\Q$ of the form \eqref{product}. Assume that $\Phi$ is in the integral closure of $(F_1, \ldots, F_m)$ and that $\supp \Phi$ satisfies \eqref{phistod}.
Then there are polynomials $G_j$ that satisfy \eqref{bsekv} 
and 
\begin{equation*}
\supp (F_j G_j) \subseteq 
\prod_{j=1}^r 
\max \big (\mu (e_j + \vol(\Q)), \min (m,n+1) d_j-n_j\big ) \Sigma^{n_j}.
\end{equation*}
\end{thm}

\begin{thm}\label{kollarprod}
Assume that $F_j$ has support in $\Q$ of the form \eqref{product}. Assume moreover that $\Phi$ vanishes on the zero set of the $F_j$ and $\supp\Phi$ satisfies \eqref{phistod}. 
Then there are polynomials $G_j$ that satisfy \eqref{kollarekv}
and 
\begin{equation}\label{prodkollarupp}
\supp (F_j G_j) \subseteq 
\prod_{j=1}^r \max\big (\mu (1 + e_j) \vol(\Q),
\min(m,n+1)d_j-n_j\big ) \Sigma^{n_j}.
\end{equation}
\end{thm}

Observe that 
\begin{equation*}
 \vol(\Q) = \vol(d_1\Sigma^{n_1}\times\cdots\times d_r \Sigma^{n_r})=
\frac{n!}{n_1! \cdot \cdots \cdot n_r!}d_1^{n_1} \cdot \cdots \cdot d_r^{n_r}.
\end{equation*}
In particular, if $n_j=1$ and $e_j=0$, then \eqref{prodkollarupp} reads that 
$\deg_{z_k}(F_jG_j) \leq  n \cdot n! \cdot d_1\cdots d_n$, which is (a slight improvement of) Rojas' example ~\cite[Example~2]{EL}. Also, observe that in general $\mu (1 + e_j) \vol(\Q)$ is much larger than $\min(m,n+1)d_j-n_j$, for example if $n_j >1$ for any $j$.

Theorems ~\ref{agprod}-\ref{kollarprod} improve Theorems \ref{agsats}-\ref{kollarsats}, respectively, for $\Q$ of the form \eqref{product}, unless $d_1=\ldots = d_r$ and $e_1=\ldots = e_r$, in which case they coincide.

Let us give a proof of Theorem \ref{hickelprod}. Theorems \ref{agprod} and \ref{kollarprod} follow along the same lines; cf. the proofs of Theorems \ref{agsats} and \ref{kollarsats}.

\begin{proof}[Proof of Theorem \ref{hickelprod}]
Note that if \eqref{phistod} holds, then, in fact, 
$\supp\Phi\subseteq \lfloor e_1\rfloor \Sigma^{n_1}\times \cdots \times \lfloor e_r \rfloor \Sigma^{n_r}$, where $\lfloor c \rfloor$ denotes the largest integer smaller than or equal to $c$. 
Let $\Q_j=\Q$, let $X=X_\Q=\P^{n_1}\times\cdots \times \P^{n_r}$, and let $\Qq= c_1 \Sigma^{n_1} \times \cdots\times c_r \Sigma^{n_r}$, where 
$c_j = \max(\mu (\lfloor e_1 \rfloor + \vol (\Q)), \min (m,n+1) d_j-n_j)$. Then 
\begin{equation*}
\O\big (D_\Qq-(D_{\Q{j_1}}+\ldots + D_{\Q_{j_{q+1}}})\big )= 
\O\big (c_1-(q+1)d_1)\boxtimes\cdots\boxtimes\O(c_r-(q+1)d_r\big ). 
\end{equation*}
Thus, by the K\"unneth Formula \eqref{kunneth}, for $q\geq 1$, \eqref{arke} is a sum  of terms which all contain a factor
\begin{equation}\label{jada}
H^{0,q}\big (\P^{n_j}, \O(c_j-(q+1)d_j)\big ),
\end{equation}
for which $q_j\geq 1$. Since $c_j \geq \min (m,n+1) d_j-n_j$, \eqref{jada} vanishes for $q\leq\min(m-1,n)$ according to Theorem \ref{forsvinnande} and so \eqref{arke} is satisfied for the required $q$ and ~$\J$. 

Let $R^f$ be the residue current associated with the $\Q$-homogenizations of the $F_j$ and let $\psi$ be the $\Qq$-homogenization of $\Phi^\mu$. By Theorem ~\ref{residysats}(c), the assumption that $\Phi\in\overline {(F)}$ implies that $\psi R^f =0$ in $\C^n$. Moreover, in light of Remark ~\ref{kvall}, $\psi$ vanishes at least to order $\mu \vol (\Q)$ along $V_\infty$, which by Lemma ~\ref{oandligheten} and Remark ~\ref{rremark} implies that $\psi$ annihilates $R^f$ at infinity. Thus $\psi$ satisfies \eqref{residy}, and so the result follows by applying Theorem ~\ref{arkesatsen}. 
\end{proof}

\section{Discussion of results}\label{exs}
Our results extend the classical results in essentially two directions. First, by taking into account the shape of the Newton polytope of the $F_j$, they give more precise estimates of $\nu$ and the degrees of the $G_j$ in \eqref{hilbert}. 
Second, our versions of Macaulay's and Max N\"other's Theorems extend the classical results in the sense that they apply to more general situations than when the $F_j$ lack common zeros at the hyperplane at infinity in $\P^n$.

\subsection{Degree estimates}\label{deges}
Our estimates of $\supp (F_j G_j)$ can be translated into degree bounds in the usual sense. Let us compare the degree estimates given by Theorem ~\ref{kollarsats} with Koll{\'a}r's result. Let $\pdeg(\poly)$ denote the degree of a generic polynomial with support in $\poly\subseteq \R^n_{\geq 0}$, in other words, $\pdeg(\poly)=\max_{\alpha\in\Z^n\cap \Q}|\alpha|$, where $|(\alpha_1,\ldots, \alpha_n)|=\alpha_1+\cdots +\alpha_n$. Then, unless $\min(m, n+1)> \lceil \mu (1/a+e) \vol (\Q)\rceil$,  \eqref{hus} gives the following degree estimate:
\begin{equation}\label{gradhus}
\deg (F_j G_j) \leq  \lceil \mu (1/a+e) \vol (\Q)\rceil \pdeg (\Q). 
\end{equation}
Assume that $\deg F_j \leq d$ and choose $\Q$ such that $\pdeg (\Q)=d$. Note that this is always possible; in particular, $\pdeg(d\Sigma^n)=d$. 
Then \eqref{gradhus} improves \eqref{kollupp} ~if
\begin{equation}\label{blak}
\vol (\Q) \leq \frac{(1+\deg \Phi) a d^{\mu-1}}{\mu (1+ ae)}; 
\end{equation}
to be precise, we should add a term $-a/(\mu(1+ae))$ to the right hand in \eqref{blak} side because of the integer parts in \eqref{gradhus}. 
Now $ae\leq \deg \Phi$ so that \eqref{blak} is in particular satisfied if 
$\vol (\Q) \leq a d^{\mu-1}/\mu$.
Thus Theorem ~\ref{kollarsats} improves Koll{\'a}r's result if the volume of the Newton polytope of the $F_j$ is small compared to $ad^{\mu-1}/\mu$, see also ~\cite{Sombra}.

An analogous analysis shows that Theorems ~\ref{agsats} and ~\ref{hickelsats} improve the results by Andersson-G\"otmark and Hickel, respectively, if $\vol (\Q) \leq ad^{\mu-1}$.

\subsection{Common zeros at infinity}\label{macmax}
Whether or not the $\Q_j$-homogenizations of the polynomials $F_j$ have common zeros at infinity clearly depends on the polytopes $\Q_j$. For example, given a smooth polytope $\Q$, the $\Q$-homogenizations of $F_j$ do have common zeros unless $\np(F_1,\ldots, F_m)=\poly $. To see this, assume that $\np(F_1,\ldots, F_m)$ is strictly included in $\poly$, so that there is a vertex $v\in \poly\setminus \np(F_1,\ldots, F_m)$. Assume that $v$ meets the facets $\tau_1, \ldots, \tau_n$ of $\Q$ with corresponding coordinates $x_1,\ldots, x_n$. That $v\notin\supp F_j$ implies that the $\poly$-homogenization $f_j$ of $F_j$ is divisible by at least one of the $x_1,\ldots, x_n$. Indeed, the $\poly$-homogenization of $z^\alpha$ where $\alpha\in \poly$ is divisible by the coordinate functions $x_i$ corresponding to the facets $\tau_i$ for which $\alpha$ is not contained in $\tau_i$. In particular, all $f_j$ vanish at the point $x_1=\ldots = x_n=0$ at infinity. 

On the other hand, the $\Q$-homogenizations of any generic choice of $n$ polynomials $F_j$ with support in $\Q$, meaning that for $\alpha\in\Q$ the coefficient of $z^\alpha$ in $F_j$ is generic, will have no common zeros at infinity, since the variety at infinity is of dimension $n-1$. 

Thus it may well happen that even though the polynomials $F_j$ have common zeros in $\P^n$ one can find a polytope $\Q$ such that the $\Q$-homogenizations (or $d_j\Q$-homogenizations) of the $F_j$ lack common zeros at infinity. Hence Theorems ~\ref{mac} and ~\ref{max} and Theorems \ref{macsamma}-\ref{maxforfinad} apply to more general systems of polynomials $F_j$ than Macaulay's and N\"other's results. Let us look at an example.

\begin{ex}\label{mellandag}
Let $F_1=z+zw+w^2$ and $F_2=z+2zw+3w^2$. Then the common zero set of $F_1$ and $F_2$ in $\C^2$ is discrete. Note that the $\P^2$-homogenizations $tz+ zw + w^2$ and $tz+2zw+3w^2$ of $F_1$ and $F_2$, respectively, have a common zero at the hyperplane at infinity, namely at $t=w=0$. Thus we cannot apply N\"other's original theorem to this example. In fact, it is not hard to check that $\Phi=z^2+2zw\in (F)$, but if $G_1, G_2$ are polynomials such that $F_1G_1+F_2G_2=\Phi$, then necessarily $\deg (F_j G_j)\geq 3$ for $j=1$ or $j=2$, so that N\"other's bound $\deg (F_j G_j)\leq \deg \Phi$ does not hold in this case. 

Let $\Q=\np(F_1,F_2,1,z,w)$, that is, the polytope with vertices $(0,0)$, $(1,0)$, $(1,1)$, and $(0,2)$. Then the corresponding toric variety $X_\Q$ is smooth and $V_\infty$ consists of two irreducible components. We choose homogeneous coordinates $z,w,x_1,x_2$ so that $\{x_1=0\}$ and $\{x_2=0\}$ are the divisors corresponding to the facets with vertices $(1,0),(1,1)$ and $(1,1),(0,2)$, respectively. According to Section ~\ref{homo},  the $\Q$-homogenizations of $F_1$ and $F_2$ are given by 
$f_1=zx_2 + zw+ w^2 x_1$ and $f_2=zx_2 + 2zw+ 3w^2 x_1$, respectively. Now $f_1$ and $f_2$ have no common zeros at $V_\infty=\{x_1=0\}\cup\{x_2=0\}$ as can be checked using local coordinates on $X_\poly$, see Section ~\ref{homo}. For example, in the $(z,1, 1, x_2)$-chart $\U$ we have that $f_1=zx_2+z+1$ and $f_2=zx_2+2z+3$ so that in $\U\cap V_\infty=\{x_2=0\}$ we get $f_1=z+1$ and $f_2=2z+3$, which clearly have no common zeros. 

It follows that we can apply Theorem ~\ref{max} to any polynomial in $(F)$. Let $\Phi=z^2 + 2 zw$. Then $\Phi\in (F)$ and $\supp \Phi \subseteq 2 \Q$, and so Theorem ~\ref{max} asserts that there are polynomials $G_j$ such that \eqref{ideal} is satisfied and $\supp (F_j G_j)\subseteq 2 \Q$. In fact, we can choose $G_1=2z+3w$ and $G_2=-z-w$. 
\end{ex}

\def\listing#1#2#3{{\sc #1}:\ {\it #2},\ #3.}


\begin{thebibliography}{9999}

\bibitem{Ahl}\listing{D.\ Ahlberg}
{Some variants of the Max N\"other and Macaulay Theorems}
{Master's Thesis, Chalmers University of Technology and G\"oteborg University, 2006}

\bibitem{A}\listing{M.\ Andersson}
{Residue currents and ideals of holomorphic functions}
{Bull.\ Sci.\ Math. {\bf 128} (2004) no. 6 481--512}


\bibitem{A3}\listing{M.\ Andersson}
{The membership problem for polynomial ideals in terms of residue currents}  
{Ann.\ Inst.\ Fourier {\bf 56}  (2006), 101--119}

\bibitem{AG}\listing{M.\ Andersson \& E.\ G{\"o}tmark}
{Explicit representation of membership of polynomial ideals} 
{Mathematische Annalen, to appear}


\bibitem{AW2}\listing{M.\ Andersson \& E.\ Wulcan}
{Decomposition of residue currents}
{J.\ Reine Angew.\ Math. {\bf 638} (2010) 103--118}

\bibitem{BGVY}\listing{C.\ A.\ Berenstein \& R.\ Gay \& A.\ Vidras \&
A.\ Yger} {Residue currents and Bezout identities}
{Progress in Mathematics
{\bf 114} Birkh\"auser Verlag (1993)}

\bibitem{BVY}\listing{C.\ A.\ Berenstein \& A.\ Vidras \& A.\ Yger}
{Analytic residues along algebraic cycles}
{J.\ Complexity  {\bf 21}  (2005),  no. 1, 5--42}


\bibitem{Bj}\listing{J-E.\ Bj\"ork}{Residues and $\mathcal D$-modules} 
{The legacy of Niels Henrik Abel,  605--651, Springer, Berlin, 2004}

\bibitem{BS}\listing{J.\ Brian\c{c}on, H.\ Skoda }
{Sur la cl\^oture int\'egrale d'un id\'eal de germes de fonctions holomorphes en un point de $\mathbb C^n$}
{C.\ R.\ Acad.\ Sci.\ Paris S\'er.\ A {\bf 278} (1974) 949--951}


\bibitem{Br}\listing{W.\ D.\ Brownawell}
{Bounds for the degrees in the Nullstellensatz}  
{Ann.\ of Math.\  {\bf 126}  (1987),  no. 3, 577--591}


\bibitem{CDV}\listing{W.\ Castryck, J.\ Denef \& F.\ Vercauteren}
{Computing zeta functions of nondegenerate curves}
{IMRP Int.\ Math.\ Res.\ Pap.\  2006, Art. ID 72017, 57 pp}

\bibitem{CH}\listing {N.\ Coleff \& M.\ Herrera}
{Les courants résiduels associcés à une forme méromorphe}{Lecture Notes in Mathematics \textbf{633} Springer Verlag, Berlin, 1978}

\bibitem{Cox}\listing{D.\ Cox}
{The homogeneous coordinate ring of a toric variety}  
{J.\ Algebraic Geom.\  {\bf 4}  (1995),  no. 1, 17--50}

\bibitem{CS}\listing{D.\ Cox \& J.\ Sidman}
{Secant varieties of toric varieties}
{J.\ Pure Appl.\ Algebra {\bf 209} (2007), no. 3, 651--669}

\bibitem{Dem}\listing{J.-P.\ Demailly}
{Complex analytic and algebraic geometry}
{Monograph, available at http://www-fourier.ujf-grenoble.fr/~demailly}

\bibitem{DS}\listing{A.\ Dickenstein  \& C.\ Sessa}
{Canonical representatives in moderate cohomology}
{Invent. Math. {\bf 80} (1985),  417--434}

\bibitem{EL}\listing{L.\ Ein \&  R.\ Lazarsfeld}
{A geometric effective Nullstellensatz}  
{Invent.\ Math.\  {\bf 137}  (1999),  427--448} 

\bibitem{F}\listing{W.\ Fulton}
{Introduction to toric varieties}
{ Annals of Mathematics Studies, 131. The William H. Roever Lectures in Geometry. Princeton University Press, Princeton, NJ, 1993}

\bibitem{H}\listing{M.\ Hickel}
{Solution d'une conjecture de C. Berenstein--A. Yger et invariants de contact à l'infini}
{Ann.\  Inst.\ Fourier  {\bf 51}  (2001), 707--744}

\bibitem{Hormander}\listing{L.\ H\"ormander}
{Generators for some rings of analytic functions}
{Bull.\ Amer.\ Math.\ Soc.\  {\bf 73} (1967) 943--949}

\bibitem{Jelonek}\listing{Z.\ Jelonek}
{On the effective Nullstellensatz}  
{Invent.\  Math.\  {\bf 162}  (2005),  no. 1, 1--17}

\bibitem{Kollar}\listing{J.\ Koll\'ar}
{Sharp effective Nullstellensatz}
{J.\ Amer.\ Math.\ Soc.\ {\bf 1} (1988), 963--975}

\bibitem{K2}\listing{J.\ Koll\'ar}
{Effective Nullstellensatz for arbitrary ideals}
{J.\ Eur.\ Math.\ Soc.\ {\bf 1} (1999), no. 3, 313--337}

\bibitem{Larkang}\listing{R.\ L\"ark\"ang}
{Residue currents associated with weakly holomorphic functions}
{Preprint, available at arXiv:0910.3589}

\bibitem{Laz}\listing{R.\ Lazarsfeld}
{Positivity in algebraic geometry. II. Positivity for vector bundles, and multiplier ideals}
{Springer-Verlag, Berlin, 2004}

\bibitem{Mac}\listing{F.\ S.\ Macaulay}
{The algebraic theory of modular systems} 
{Cambridge University Press, Cambridge, 1916}

\bibitem{Not}\listing{M.\ N\"other}
{\"Uber einen Satz aus der Theorie der algebraischen Functionen} 
{Math.\ Ann.\ {\bf 6} (1873), no. 3, 351--359}

\bibitem{Oda}\listing{T.\ Oda}
{Convex bodies and algebraic geometry. An introduction to the theory of toric varieties, Ergebnisse der Mathematik und ihrer Grenzgebiete 15}
{Springer-Verlag, Berlin, 1988}

\bibitem{P}\listing{M.\ Passare}
{Residues, currents, and their relation to ideals of holomorphic functions} 
{Math.\ Scand.\ {\bf 62} (1988), no. 1, 75--152}

\bibitem{PTY}\listing{M.\ Passare \& A.\ Tsikh \&  A.\ Yger}
{Residue currents of the Bochner-Martinelli type}
{Publ.\ Mat.  {\bf 44} (2000), 85--117}

\bibitem{Payne}\listing{S.\ Payne}
{Fujita's very ampleness conjecture for singular toric varieties}
{Tohoku Math.\ J.\  {\bf 58}  (2006),  no. 3, 447--459}

\bibitem{Skoda}\listing{H.\ Skoda}
{Application des techniques $L\sp{2}$ à la théorie des idéaux d'une algèbre de fonctions holomorphes avec poids}
{Ann.\ Sci.\ École Norm.\ Sup.\ {\bf 5}  (1972), 545--579}

\bibitem{Sombra}\listing{M.\ Sombra}
{A sparse effective Nullstellensatz}
{Adv.\ in Appl.\ Math.\ {\bf 22} (1999) 271--295}

\bibitem{Sombra2}\listing{M.\ Sombra}
{Private communication}
{}

\bibitem{Tuitman}\listing{J.\ Tuitman}
{A mixed sparse effective nullstellensatz}
{Preprint, 2008}

\bibitem{Weimann}\listing{M.\ Weimann}
{La trace en g\'eom\'etrie projective et torique}
{Ph.\ D.\ Thesis, Universit\'e Bordeaux, 2006}

\bibitem{W}\listing{E.\ Wulcan}
{Some variants of Macaulay's and Max Noether's Theorems}
{Journal of Commutative Algebra, to appear}

\end{thebibliography}
\end{document}